\documentclass[reqno,12pt]{amsart}
\newcommand{\CC}{{\mathbb C}}

\def\bege{\begin{equation}} \def\ende{\end{equation}}
\def\begr{\begin{eqnarray}} \def\endr{\end{eqnarray}}
\usepackage{amsmath}
\usepackage{amssymb}
\usepackage{cite}

\usepackage{bbm}
\usepackage{amsmath}
\usepackage{txfonts}
\usepackage{stmaryrd}
\usepackage{amssymb}
\usepackage{amsfonts}
\usepackage{times}
\usepackage{mathrsfs}
\usepackage{soul}

\usepackage{amsthm}
\usepackage{enumerate}

\usepackage{framed}
\usepackage{lipsum}
\usepackage{color}

\def\CC{ \mathbb{C}}
\newcommand{\DD}{{\mathbb D}}

\def\R{\mathcal{R}}

\def\T{\mathcal{T}}

\def\D{\mathbb{D}}
\def\N{\mathbb N}

\def\hD{\hat{\mathcal{D}}}

\def\vp{\varphi}
\def\om{\omega}

\def\p{{\prime}}

\def\begr{\begin{eqnarray}} \def\endr{\end{eqnarray}}

\def\msk{\medskip}
\def\ol{\overline}
\newtheorem{Lemma}{Lemma}
\newtheorem{Theorem}{Theorem}

\newtheorem*{th1p}{Theorem \ref{th1}$^\p$}

\newcounter{other}            
\newtheorem{otherth}[other]{Theorem}
\begin{document}
\title[]{ Toeplitz operators and Carleson measure  between weighted Bergman spaces induced by regular weights}

\author{ Juntao Du,     Songxiao Li$\dagger$ and Hasi Wulan}

\address{Juntao Du\\ Department of mathematics, Shantou University, Shantou, Guangdong,  515063, China.}
\email{jtdu007@163.com  }

\address{Songxiao Li\\ Institute of Fundamental and Frontier Sciences, University of Electronic Science and Technology of China,
610054, Chengdu, Sichuan, P.R. China. }
\email{jyulsx@163.com}

\address{Hasi Wulan\\ Department of mathematics, Shantou University, Shantou, Guangdong,  515063, China.}
\email{wulan@stu.edu.cn  }

 \begin{abstract}
 In this paper, we give a universal description of  the boundedness and compactness of Toeplitz operator $\mathcal{T}_\mu^\omega$ between   Bergman spaces $A_\eta^p$ and $A_\upsilon^q$ when $\mu$ is a positive Borel measure,  $1<p,q<\infty$ and $\omega,\eta,\upsilon$ are   regular  weights.
By using Khinchin's inequality and Kahane's inequality, we get a new characterization of the Carleson measure for   Bergman spaces induced by regular weights.
 \thanks{$\dagger$ Corresponding author.}
 \thanks{The first author is supported by NSF of Guangdong Province (No. 2022A1515012117) and the corresponding author is supported by NSF of Guangdong Province (No. 2022A1515010317).}
 \end{abstract}  \maketitle

\section{Introduction}

Suppose $\om$ is a radial weight (i.e., $\om$ is a positive, measurable and  integrable  function on $[0,1)$ and $\om(z)=\om(|z|)$ for all $z\in\D$, the open unit disc in the complex plane).   We say that  $\om$ is a doubling weight, denoted by $\om\in \hD$,  if $$\hat{\om}(r)\lesssim \hat{\om}(\frac{r+1}{2})$$ for   all $r\in [0,1)$, where  $\hat{\om}(r)=\int_r^1\om(t)dt$.
   $\om$ is a regular weight,  denoted by  $\om\in    \R $, if 
$$\frac{\hat{\om}(r)}{(1-r)\om(r)}\approx 1$$ for   all  $r\in [0,1)$.  Obviously, if $\om\in\R$,  $\om\in \hD$ and $\om(t)\approx \om(s)$   whenever $ s,t\in(0,1)$ satisfying  $1-t\approx 1-s$.   See \cite{Pja2015,PjaRj2014book,PjaRjSk2018jga,PjRj2021adv} for more information about   doubling weights and related topics.

For   $z\in\D$, the Carleson square $S_z$ is defined by
$$S_z=\left\{\xi=re^{i\theta}\in\D: |z|\leq r<1; |\theta-\arg z|<\frac{1-|z|}{2\pi}\right\}.$$
As usual, let $\beta(a,z)=\frac{1}{2}\log\frac{1+|\vp_a(z)|}{1-|\vp_a(z)|}$ denote  the Bergman metric  for $a,z\in\D$ and
 $$D(z,r)=\big\{a\in\D:  ~~\beta(z,a)<r\big\}$$
be the Bergman disk with  radius $r$ and center $z$. Here $\vp_a(z)=\frac{a-z}{1-\ol{a}z}$.
It is easy to check that, for any given $\om\in\R$ and $r\in (0,\infty)$,
\begin{align}\label{0410-1}
\om(D(z,r))\approx \om(S_z)\approx (1-|z|)^2\om(z),\,\, \,\,z\in\D.
\end{align}
For $0<s <\infty$, a sequence $\{a_j\}_{j=1}^\infty\subset\DD$ is  called $s$-separated(or separated, for brief) if $\beta(a_i,a_j)\geq s$ for all $i\neq j$.  A sequence $\{a_j\}_{j=1}^\infty\subset\DD$ is  called  $r$-covering  if  $\D=\cup_{j=1}^\infty D(a_j,r)$ for some $ r\in (0,\infty)$.

Let  $0<p<\infty$ and $\mu$ be a positive Borel measure on $\D$. The Lebesgue space  $L_\mu^p$
 consists of all measurable functions $f$ on $\D$ such that
$$\|f\|_{L_\mu^p}^p=\int_\D  |f(z)|^p d\mu(z) <\infty.$$
In particular, when $d\mu(z)=\om(z)dA(z)$, we denote $L_\mu^p$ by $L^p_\om$. Here $dA $ is the normalized area measure on $\D$.

Let  $0<p<\infty$ and   $\om$ be a  radial weight. Let $H(\D)$ denote  the space of all analytic functions in $\D$. The weighted Bergman space $A_\om^p $  consists of all $f\in H(\D)$  such that
$$\|f\|_{A_\om^p}^p=\int_\D |f(z)|^p\om(z)dA(z)<\infty.
$$
Throughout this paper, we assume that $\hat{\om}(z)>0$ for all $z\in\D$. For otherwise $A_\om^p=H(\D)$.
As usual, we write $A_\alpha^p$ for the standard weighted Bergman space induced by  $\om(z)=(\alpha+1)(1 - |z|^2)^\alpha$ with $-1<\alpha<\infty$.
 For convenience, the weight  $(\alpha+1)(1-|z|^2)^\alpha(-1<\alpha<\infty)$ will be called standard weight.

Let  $\om\in\hD$. It is easy to see that  $A_\om^2$ is a closed
subspace of $L_\om^2$ and the orthogonal Bergman projection  $P_\om$ from  $L_\om^2$ to $A_\om^2$ is given by
$$P_\om f(z)=\int_\D f(\xi)\overline{B_z^\om(\xi)}\om(\xi) dA(\xi),$$
where  $B_z^\om$ is the reproducing kernel of $A_\om^2$.
For more results about the Bergman projection $P_\om$, see  \cite{PjRj2016jmpa,PjRj2021adv}.
For a positive Borel measure $\mu$ on $\D$, the Toeplitz operator associated with $\mu$,   a natural extension of $P_\om$, is defined by $$\T_\mu^\om f(z)=\int_\D f(\xi)\overline{B_z^\om(\xi)}d\mu(\xi).$$
For the standard weight, we will write $P_\om$ and $\T_\mu^\om$ as $P_\alpha$ and $\T_\mu^\alpha$, respectively.

Toeplitz operators on Bergman spaces attract a lot of attentions in the past decades. See \cite[Chapter 7]{zhu} and the references therein for
the theory of Toeplitz operators on  standard weighted Bergman spaces $A^p_\alpha$. In \cite{PjaRj2016}, Pel\'aez and R\"atty\"a  characterized the Schatten class Toeplitz operators $\T_\mu^\om(\om\in\hD)$ on $A_\om^2$.  Pel\'aez, R\"atty\"a  and Sierra\cite{PjaRjSk2018jga} gave several characterizations of bounded and compact  Toeplitz operators $\T_\om^p:A_\om^p\to A_\om^q$ when $\om\in\R$ and $1<p,q<\infty$.
Recently, Duan, Guo, Wang and Wang \cite{DGWW} extended these results to $0<p\leq 1$ and $q=1$. 
See\cite{DW2022, DGWW, Ls2022jmaa, LHT, Lv, ZrZk} for more discussions on this topic.

It is interesting to study the boundedness and compactness of $\T_\mu^\om:A_\eta^p\to A_\upsilon^q$ when $\om,\eta,\upsilon$ are   different     weights.
When $\om,\eta,\upsilon$ are all standard weights, this problem was studied by Pau and Zhao in \cite{PjZr2015mmj}  in the case of the unit ball. Motivated by \cite{PjZr2015mmj},
  we study the boundedness and compactness of $\T_\mu^\om:A_\eta^p\to A_\upsilon^q$ when $1<p,q<\infty$ and $\om,\eta,\upsilon$ are   regular  weights.   Set
\begin{align}\label{0315-1}
\sigma_{p,\eta}(r)=\left(\frac{\om(r)}{\eta(r)^{\frac{1}{p}}}\right)^{p^\p}\,\,\,\,\,\mbox{ and }\,\,\,\,
A(p,\eta)=\sup_{0\leq r<1}\frac{\widehat{\eta}(r)^\frac{1}{p} \widehat{\sigma_{p,\eta}}(r)^\frac{1}{p^\p}}{\widehat{\om}(r)}.
\end{align}
Here $p^\p=\frac{p}{p-1}$ is the conjugate number of $p$. The first result of this paper is stated as follows.

\begin{Theorem}\label{th1}
Suppose $1<p,q <\infty$, $\mu$ is a positive Borel measure on $\D$ and $\om,\eta,\upsilon\in\R$ such that
\begin{align*}
A(p,\eta)<\infty,\,\,\,\,\, A(q,\upsilon)<\infty.
\end{align*}
Let
$$W(z)=\eta(z)^\frac{q}{pq-p+q}\sigma_{q,\upsilon}(z)^\frac{pq-p}{pq-p+q}.$$
Then, $\T_\mu^\om:A_\eta^p\to A_\upsilon^q$ is bounded (compact) if and only if
$\mu$ is a  1-Carleson measure (vanishing 1-Carleson measure) for $A_{W}^{\frac{pq}{pq-p+q}}$. Moreover,
\begin{align*}
\|\T_\mu^\om\|_{A_\eta^p\to A_\upsilon^q}  \approx \|I_d\|_{A_{W}^{\frac{pq}{pq-p+q}}\to L_\mu^1}.
\end{align*}
\end{Theorem}

For a function space $X$ and $0<q<\infty$, a positive measure $\mu$ on $\D$ is called a $q$-Carleson measure (vanishing $q$-Carleson measure) for $X$ if the identity operator $I_d:X\to L_\mu^q$ is bounded (compact).
 When  $0<p,q<\infty$, the $q$-Carleson measure  for $A_\alpha^p$ was characterized by many authors,  we refer to \cite{ZrZk, zhu} and the references therein. When $\om\in\hD$, the problem was completely solved by Pel\'aez and R\"atty\"a in \cite{PjaRj2014book,PjRj2015}.
  See \cite{Co2010jmaa,Lv, PjaRjSk2015mz,PjaRjSk2018jga} for more study of Carleson measure for Bergman spaces induced by various weights.

In \cite{PjZr2015mmj}, Pau and Zhao gave a new characterization of Carleson measures for standard weighted Bergman spaces in the unit ball by using the technique of sublinear operator and   the characterizations of the boundedness and compactness of $\T_\mu^\alpha:A_\beta^p\to A_\gamma^q$.
Here, we extend their result (i.e., \cite[Theorem 1.1]{PjZr2015mmj}) to  Bergman spaces induced by regular weights in a more direct way, without the using of Theorem   \ref{th1}.
    We state it as follows.

\begin{Theorem}\label{th2}
Suppose  $0<p_i,q_i<\infty$, $\om_i\in\R$, $i=1,2,\cdots,n$ and $\mu$ is a positive Borel measure on $\D$.
Let
$$\lambda=\sum_{i=1}^n \frac{q_i}{p_i},\,\,\,\,\,\, \om(z)=\prod\limits_{i=1}^n \om_i(z)^\frac{q_i}{\lambda p_i}.$$
Then $\mu$ is a $\lambda$-Carleson measure for $A_\om^1$ if and only if
\begin{align}\label{th2-r}
M_n=\sup_{f_i\in A_{\om_i}^{p_i},i=1,2,\cdots,n}\frac{\int_\D \prod_{i=1}^n |f_i(z)|^{q_i} d\mu(z)}{\prod_{i=1}^n \|f_i\|_{A_{\om_i}^{p_i}}^{q_i}}<\infty.
\end{align}
Moreover,
$M_n\approx \|I_d\|_{A_\om^{1/\lambda}\to L_\mu^1}.$
\end{Theorem}

The analogous characterizations of vanishing $\lambda$-Carleson measures for $A_\om^1(\om\in\R)$ can be state as follows.  \msk

\begin{Theorem}\label{th3}
Suppose $0<p_i,q_i<\infty$, $\om_i\in\R$, $i=1,2,\cdots,n$  and $\mu$ is a positive Borel measure on $\D$.
Let
$$\lambda=\sum_{i=1}^n \frac{q_i}{p_i},\,\,\,\,\,\, \om(z)=\prod\limits_{i=1}^n \om_i(z)^\frac{q_i}{\lambda p_i}.$$
Then the following statements are equivalent.
\begin{enumerate}[(i)]
  \item $\mu$ is a vanishing $\lambda$-Carleson measure for $A_\om^1$;
  \item If $\{f_{1,k}\}_{k=1}^\infty$ is bounded in $A_{\om_1}^{p_1}$ and converges to 0 uniformly on compact subsets of\, $\D$,
   $$\lim_{k\to \infty} F(k)=0,$$
   where
   $$F(k)=\sup\left\{\int_\D |f_{1,k}(z)|^{q_1}\prod_{i=2}^n |f_i(z)|^{q_i}d\mu(z):\|f_i\|_{A_{\om_i}^{p_i}}\leq 1,i=2,\cdots,n\right\};$$
  \item For any bounded sequences $\{f_{1,k}\}_{k=1}^\infty$, $\cdots$, $\{f_{n,k}\}_{k=1}^\infty$ in $A_{\om_1}^{p_1}$, $\cdots$, $A_{\om_2}^{p_n}$, respectively, all of which  converge to 0 uniformly on compact subsets of $\D$,
      $$\lim_{k\to \infty}\int_{\D} \prod_{i=1}^n |f_{i,k}(z)|^{q_i} d\mu(z)=0.$$
\end{enumerate}
\end{Theorem}

The paper is organized as follows. In Section 2, we state some preliminary results and lemmas, which will be used later.
Section 3 is devoted to the proof of main results in this paper.

Throughout this paper,  the letter $C$ will denote  constants and may differ from one occurrence to the other.
For two real valued functions $f$ and $g$, we write   $f \lesssim g$ if there is a positive constant $C$, independent of argument,  such that $f\leq Cg$.   $f\approx g$ means that $f \lesssim g$  and $g \lesssim f$.

\section{Preliminaries}

To prove our main results in this paper, we need some lemmas.  First, we collect some characterizations of $q$-Carleson measure for $A_\om^p$ $(\om\in\R)$ from  \cite{Co2010jmaa, PjaRjSk2015mz,PjaRjSk2018jga}   as follows.

\noindent\begin{otherth}\label{thA}
Let $0<p,q<\infty$, $\om\in\R$ and $\mu$ be a positive Borel measure on $\D$.
Then, the following statements hold.
\begin{enumerate}[(i)]
  \item When $p\leq q$,   the following statements are equivalent:
  \begin{enumerate}
    \item[(ia)] $\mu$ is a $q$-Carleson measure for $A_\om^p$;
    \item[(ib)] $\sup\limits_{z\in\D}\frac{\mu(S_z)}{\om(S_z)^\frac{q}{p}}<\infty$;
    \item[(ic)] $\sup\limits_{z\in\D}\frac{\mu(D(z,r))}{\om(D(z,r))^\frac{q}{p}}<\infty$ for some (equivalently, for all) $r\in (0,\infty)$.
  \end{enumerate}
   Moreover,
   $$\|I_d\|_{A_\om^p\to L_\mu^q}^q\approx \sup_{z\in\D}\frac{\mu(S_z)}{\om(S_z)^\frac{q}{p}} \approx \sup_{z\in\D}\frac{\mu(D(z,r))}{\om(D(z,r))^\frac{q}{p}}.$$
    \item When $p\leq q$,   the following statements are equivalent:
          \begin{enumerate}
          \item[(iia)] $\mu$ is a vanishing $q$-Carleson measure for $A_\om^p$;
          \item[(iib)] $\lim\limits_{|z|\to 1}\frac{\mu(S_z)}{\om(S_z)^\frac{q}{p}}=0$;
          \item[(iic)] $\lim\limits_{|z|\to 1}\frac{\mu(D(z,r))}{\om(D(z,r))^\frac{q}{p}}=0$ for some (equivalently, for all) $r\in (0,\infty)$.
          \end{enumerate}
  \item When $q<p$, the following statements are equivalent:
          \begin{enumerate}
          \item[(iiia)] $\mu$ is a $q$-Carleson measure for $A_\om^p$;
          \item[(iiib)] $\mu$ is a vanishing $q$-Carleson measure for $A_\om^p$;
          \item[(iiic)] for any given $r\in (0,\infty)$, $\Phi(z)=\frac{\mu(D(z,r))}{\om(D(z,r))}\in L_\om^\frac{p}{p-q}$;
          \item[(iiid)] for each sufficiently large $\gamma>1$, the function
               $$\Psi(z)=\int_\D \left(\frac{1-|\xi|}{|1-\overline{z}\xi|}\right)^\gamma \frac{d\mu(\xi)}{\om(S_\xi)},\,\,\,z\in\D, $$
                belongs to $L_\om^\frac{p}{p-q}$.
          \end{enumerate}
  Moreover,
  $$\|I_d\|_{A_\om^p\to L_\mu^q}^q \approx \|\Phi\|_{L_\om^\frac{p}{p-q}} \approx \|\Psi\|_{L_\om^\frac{p}{p-q}}.$$
\end{enumerate}
\end{otherth}

 The duality relation between weighted Bergman space via $A_\om^2$-pairing  plays a very important role in the proof of our main results in this paper.
  See \cite{Zk1994} and \cite[Theorem 2.11]{zhu2} for related results for standard weighted Bergman spaces.
For Bergman spaces induced by radial weights, we refer to \cite{PjRj2021adv,PaRj2017,RfOpPjRj2021}.
The following theorem comes from \cite[Theorem 13]{PjRj2021adv} and \cite[Theorem 3]{RfOpPjRj2021}.

\begin{otherth}\label{thB}
Suppose $1<p<\infty$ and $\om,\eta\in\R$.  Then the following statements are equivalent:
\begin{enumerate}[(i)]
  \item $A(p,\eta)<\infty$;
  \item $P_\om:L_\eta^p\to L_\eta^p$ is bounded;
  \item $(A_{\sigma_{p,\eta}}^{p^\p})^*\simeq A_\eta^p$ via $A_\om^2$-pairing
  $$\langle f,g\rangle_{A_\om^2}=\int_\D f(z)\ol{g(z)}\om(z)dA(z),\,\,\,\,\forall\, f\in A_{\sigma_{p,\eta}}^{p^\p}, \,\,
  g\in A_\eta^p,$$
  with equivalent norms. Here $\sigma_{p,\eta}$ and $A(p,\eta)$ are defined as in (\ref{0315-1}).
\end{enumerate}
\end{otherth}

The following lemma can be proved by  a standard technique. We omit the details of the proof. \msk

\begin{Lemma}\label{lm-compact}
Suppose $\eta,\upsilon\in\hD$ and $0<p,q<\infty$. If $T:A_\eta^p\to A_\upsilon^q$ is bounded and linear, then $T$ is compact
if and only if $\lim\limits_{j\to\infty}\|Tf_j\|_{A_\upsilon^q}=0$ whenever $\{f_j\}$ is bounded in $A_\eta^p$ and converges  to 0 uniformly on  compact subsets of $\D$.
\end{Lemma}

\begin{Lemma}\label{lm3}
Let $1<p<\infty$ and $\om,\eta\in\R$. Then,  $A(p,\eta)<\infty$ if and only if $\sigma_{p,\eta}\in\R$.
\end{Lemma}

\begin{proof}
If $\sigma_{p,\eta}\in\R$, from the fact that $$\widehat{\sigma_{p,\eta}}(r)\approx (1-r)\sigma_{p,\eta}(r),$$  we have
$A(p,\eta)<\infty$.

Conversely, assume $A(p,\eta)<\infty$.  After a calculation,
$$
\widehat{\sigma_{p,\eta}}(r)
\leq A(p,\eta)^{p^\p} \frac{\widehat{\om}(r)^{p^\p}}{\widehat{\eta}(r)^{p^\p/p}}
\approx A(p,\eta)^{p^\p} (1-r)\sigma_{p,\eta}(r) $$
and
$$
\widehat{\sigma_{p,\eta}}(r)
\geq \int_r^\frac{1+r}{2}\sigma_{p,\eta}(t)dt
\approx (1-r)\sigma_{p,\eta}(r),
$$
which implies that $\sigma_{p,\eta}\in\R$.
\end{proof}

\begin{Lemma}\label{lm2}
Let $1<p<\infty$ and $\om,\eta\in\R$ such that $A(p,\eta)<\infty$.
If $\{z_j\}_{j=1}^\infty\subset\D$ is separated,  then, for any $c=\{c_j\}_{j=1}^\infty$,
 $$\|F\|_{A_\eta^p}\lesssim \|c\|_{l^p}.$$
Here
$$F(z)=\sum_{j=1}^\infty \frac{c_j B_{z_j}^{\om}}{\|B_{z_j}^{\om}\|_{A_\eta^p}}.$$
\end{Lemma}
\begin{proof}
Since $A(p,\eta)<\infty$,
by Theorem \ref{thB},   $(A_{\sigma_{p,\eta}}^{p^\p})^*\simeq A_\eta^p$  via $A_\om^2$-pairing with equivalent norms.
By Lemma \ref{lm3}, $\sigma_{p,\eta}\in\R$.
Then Lemma 2.4 in \cite{PjaRj2014book} implies that, if $\gamma$ is large enough,
$$F_a(z)=\left(\frac{1-|a|^2}{1-\overline{a}z}\right)^\gamma\in A_{\sigma_{p,\eta}}^{p^\p}~~~\, \mbox{ and}~~\,  \|F_a\|_{A_{\sigma_{p,\eta}}^{p^\p}}^{p^\p}\approx \sigma_{p,\eta}(S_a).$$
Then,
$$
\|B_z^\om\|_{A_\eta^p}
\approx\sup_{\|h\|_{A_{\sigma_{p,\eta}}^{p^\p}}\leq 1}|\langle h, B_z^\om\rangle_{A_\om^2}|
\gtrsim \frac{1}{\sigma_{p,\eta}(S_z)^{1/{p^\p}}} |\langle F_z, B_z^\om\rangle_{A_\om^2}|
=\frac{1}{\sigma_{p,\eta}(S_z)^{1/{p^\p}}}
$$
and
\begin{align*}
\|B_z^\om\|_{A_\eta^p}
\approx\sup_{\|h\|_{A_{\sigma_{p,\eta}}^{p^\p}}\leq 1}|\langle h, B_z^\om\rangle_{A_\om^2}|
=\sup_{\|h\|_{A_{\sigma_{p,\eta}}^{p^\p}}\leq 1} |h(z)|\lesssim \frac{1}{\sigma_{p,\eta}(S_z)^{1/{p^\p}}}.
 \end{align*}
In the last estimate, we used the fact that   $\sigma_{p,\eta}\in\R$, (\ref{0410-1}) and the  subharmonicity of $|h|^{p^\p}$, i.e., for any given $0<r<\infty$,
\begin{align*}
|h(z)|^{p^\p}
&\lesssim \frac{1}{\sigma_{p,\eta}(D(z,r))}\int_{D(z,r)} |h(\xi)|^{p^\p}\sigma_{p,\eta}(\xi)dA(\xi)\\
&\lesssim \frac{1}{\sigma_{p,\eta}(S_z)}\|h\|_{A_{\sigma_{p,\eta}}^{p^\p}}^{p^{\p}},~~~~z\in\D.
\end{align*}
Thus, by the fact that $\sigma_{p,\eta}\in\R$  we have
\begin{align}\label{th1-2}
\|B_z^\om\|_{A_\eta^p}
\approx \frac{1}{\sigma_{p,\eta}(S_z)^{1/{p^\p}}}
\approx \frac{\eta(S_z)^{1/p}}{\om(S_z)}.
\end{align}
So, if $\{z_j\}$ is $2r$-separated, for any $g\in A_{\sigma_{p,\eta}}^{p^\p}$, by H\"older's inequality, (\ref{0410-1})  and (\ref{th1-2}), we have
\begin{align*}
\left| \langle g,F\rangle_{A_{\om}^2}\right|
&=\left|
\sum_{j=1}^\infty \frac{c_j \langle g,B_{z_j}^{\om}\rangle_{A_{\om}^2}}{\|B_{z_j}^{\om}\|_{A_\eta^p}}\right|
=\left| \sum_{j=1}^\infty \frac{c_j g(z_j)}{\|B_{z_j}^{\om}\|_{A_\eta^p}}\right|   \\
&\leq \|c\|_{l^p}  \left( \sum_{j=1}^\infty \frac{ |g(z_j)|^{p^\p}}{\|B_{z_j}^{\om}\|_{A_\eta^p}^{p^\p}}\right)^\frac{1}{p^\p}
\approx \|c\|_{l^p}  \left( \sum_{j=1}^\infty { {\sigma_{p,\eta}(D(z_j,r))} \cdot  |g(z_j)|^{p^\p}}\right)^\frac{1}{p^\p} \\
&\lesssim  \|c\|_{l^p}  \left( \sum_{j=1}^\infty \int_{D(z_j,r)} |g(z)|^{p^\p}\sigma_{p,\eta}(z)dA(z)\right)^\frac{1}{p^\p} \\
&\lesssim \|c\|_{l^p}  \|g\|_{ A_{\sigma_{p,\eta}}^{p^\p}}.
\end{align*}
Therefore,
$\|F\|_{A_{\eta}^p}\lesssim \|c\|_{l^p}.$
The proof is complete.
\end{proof}

\noindent{\bf Remark.} It  should be point out that, when $\eta,\om\in\hD$ and $p>0$,   $\|B_z^\om\|_{A_\eta^p}$ was  estimated  for the  first time in \cite[Theorem 1]{PjRj2016jmpa}.
Here, for the benefit of readers, under some more assumptions,  we estimate it in a simple way.\msk

To prove our main results,  we will use the classical Khinchin's inequality and Kahane's inequality, which can be found in  \cite[Chapters 1 and 11]{DjThTa} for example.
For $k\in\N$ and $t\in(0,1)$, let $r_k(t)=$sign$(\sin (2^k\pi t))$ be  a sequence of Rademacher functions.  \msk

\noindent{\bf Khinchin's inequality: } Let $0<p<\infty$. Then for any sequence $\{c_k\}$ of complex numbers,
$$\left(\sum_{k=1}^\infty |c_k|^2\right)^\frac{1}{2}\approx \left(\int_0^1 \left|\sum_{k=1}^\infty c_kr_k(t) \right|^pdt\right)^\frac{1}{p}.
$$

\noindent{\bf Kahane's inequality: } Let $X$ be a quasi-Banach space, and $0<p,q<\infty$. For any sequence $\{x_k\}\subset X$,
$$\left(\int_0^1 \left\|\sum_{k=1}^\infty r_k(t)x_k\right\|_X^p dt\right)^\frac{1}{p}
\approx \left(\int_0^1 \left\|\sum_{k=1}^\infty r_k(t)x_k\right\|_X^q dt\right)^\frac{1}{q}.$$
Moreover, the implicit constants   can be chosen to depend only on $p$ and $q$, and  independent of the  space $X$.\msk

\section{Proof of main results}

By Theorem \ref{thA}, Theorem \ref{th1} can be stated in a more direct way as follows. Hence, to prove Theorem \ref{th1} we only need to prove the following theorem.

\begin{th1p}
Suppose $1<p,q <\infty$, $\mu$ is a positive Borel measure on $\D$ and $\om,\eta,\upsilon\in\R$ such that
\begin{align}\label{th1-c1}
A(p,\eta)<\infty,\,\,\,\,\, A(q,\upsilon)<\infty.
\end{align}
Then, the following statements hold.
\begin{enumerate}[(i)]
  \item When $1<p\leq q<\infty$,  $\T_\mu^\om:A_\eta^p\to A_\upsilon^q$ is bounded if and only if  for some (equivalently, for all) $r\in (0,\infty)$,
 \begin{align}\label{th1-r1}
M_0=\sup_{z\in\D}\frac{\mu(D(z,r))}{\om((D(z,r)))}   \frac{\upsilon(D(z,r))^\frac{1}{q}}{\eta(D(z,r))^\frac{1}{p}}<\infty.
\end{align}
Moreover, $\|\T_\mu^\om\|_{A_\eta^p\to A_\upsilon^q}\approx M_0.$
  \item When $1<p\leq q<\infty$,
  $\T_\mu^\om:A_\eta^p\to A_\upsilon^q$ is compact if and only if for some (equivalently, for all) $r\in (0,\infty)$,
 \begin{align}\label{th1-r2}
\lim_{|z|\to 1}\frac{\mu(D(z,r))}{\om((D(z,r)))}   \frac{\upsilon(D(z,r))^\frac{1}{q}}{\eta(D(z,r))^\frac{1}{p}}=0.
\end{align}
\item When $1<q< p<\infty$, the following statements are equivalent:
\begin{enumerate}
  \item[(iiia)] $\T_\mu^{\om}:A_{\eta}^p\to A_{\upsilon}^q$ is compact;
  \item[(iiib)] $\T_\mu^{\om}:A_{\eta}^p\to A_{\upsilon}^q$ is bounded;
  \item[(iiic)] for some (equivalently, for all)  separated and $r$-covering($0<r<\infty$) sequence $\{z_j\}_{j=1}^\infty$,
      $$\lambda_j=\frac{\mu(D(z_j,r))}{\om(D(z_j,r))}
         \frac{\upsilon(D(z_j,r))^\frac{1}{q}}{\eta(D(z_j,r))^\frac{1}{p} },\,\,\, j=1,2,\cdots, $$
         is a sequence in $ l^\frac{pq}{p-q}$;
  \item[(iiid)] for some (equivalently, for all) $r\in (0,\infty)$, $\widehat{\mu}_r\in L_\om^\frac{pq}{p-q}$, where
  \begin{align*}
  \widehat{\mu}_r(z)=\frac{\mu(D(z,r))}{\om(D(z,r))^{1+\frac{1}{q}-\frac{1}{p}}}
         \frac{\upsilon(D(z,r))^\frac{1}{q}}{\eta(D(z,r))^\frac{1}{p} }.
  \end{align*}
\end{enumerate}
Moreover, $$\|\T_\mu^\om\|_{A_\eta^p\to A_\upsilon^q} \approx \|\{\lambda_j\}\|_{l^\frac{pq}{p-q}}\approx \|\widehat{\mu}_r\|_{L_\om^\frac{pq}{p-q}}.$$
\end{enumerate}
\end{th1p}

\begin{proof}
For convenience, write $\|\T_\mu^\om\|=\|\T_\mu^\om\|_{A_\eta^p\to A_{\upsilon}^q}$.
Since (\ref{th1-c1}) holds, by  Lemma \ref{lm3} and Theorem \ref{thB}, we have that
$\sigma_{p,\eta}\in\R$, $\sigma_{q,\upsilon}\in\R$ and
\begin{align}\label{th1-1}
(A_{\sigma_{p,\eta}}^{p^\p})^*\simeq A_\eta^p,\,\,\,\,\,\,\,(A_{\sigma_{q,\upsilon}}^{q^\p})^*\simeq A_\upsilon^{q}
\end{align}
 via $A_\om^2$-pairing with equivalent   norms.

$(i)$. Suppose $1<p\leq q<\infty$ and $\T_\mu^{\om}:A_{\eta}^p\to A_{\upsilon}^q$ is bounded.
By (\ref{th1-2}), we have
\begin{align*}
|\T_\mu^{\om} B_z^{\om}(z)|
=\left|\langle \T_\mu^{\om} B_z^{\om}, B_z^{\upsilon}   \rangle_{A_{\upsilon}^2}\right|
&\leq \|\T_\mu^\om B_z^\om\|_{A_\upsilon^q}\|B_z^\upsilon\|_{A_\upsilon^{q^\p}}   \\
&\leq \|\T_\mu^{\om}\|  \cdot \| B_z^{\om}\|_{A_\eta^p}  \|B_z^\upsilon\|_{A_\upsilon^{q^\p}}\\
&\approx \|\T_\mu^{\om}\|  \frac{\eta(S_z)^\frac{1}{p}}{\om(S_z)}  \frac{1}{\upsilon(S_z)^\frac{1}{q}}.
\end{align*}
Meanwhile, by \cite[Lemma 8]{PjaRjSk2018jga}, there exists a real number $r>0$ such that
\begin{align}\label{0315-2}
|B_z^{\om}(\xi)|\approx B_z^\om(z) \,\,\mbox{ for all }\,\,\xi\in D(z,r)\,\,\mbox{ and } z\in \D.
\end{align}
Then, (\ref{th1-2}) implies
$$|B_z^\om(z)|=\langle B_z^\om, B_z^\om\rangle_{A_\om^2}=\|B_z^\om\|_{A_\om^2}^2\approx \frac{1}{\om(S_z)}.$$
Thus,
\begin{align}
(\T_\mu^{\om} B_z^{\om})(z)
&=\int_\D |B_z^{\om}(\xi)|^2 d\mu(\xi)
\gtrsim  \frac{\mu(D(z,r))}{\om(S_z)^2}.  \label{0410-3}
\end{align}
Therefore, using (\ref{0410-1}),
$$
M_0\approx \sup_{z\in\D}\frac{\mu(D(z,r))}{\om(S_z)}   \frac{\upsilon(S_z)^\frac{1}{q}}{\eta(S_z)^\frac{1}{p}}\lesssim \|\T_\mu^\om\|.$$

Conversely, suppose  (\ref{th1-r1}) holds. Let
\begin{align}\label{0317-1}
W(z)=\eta(z)^\frac{q^\p}{p+q^\p}\sigma_{q,\upsilon}(z)^\frac{p}{p+q^\p}.
\end{align}
By H\"older's inequality and $\sigma_{q,\upsilon}\in\R$,
\begin{align*}
\widehat{W}(t)
\leq \widehat{\eta}(t)^\frac{q^\p}{p+q^\p}\widehat{\sigma_{q,\upsilon}}(t)^\frac{p}{p+q^\p}
\approx (1-t)W(t),
\end{align*}
and
$$\widehat{W}(t)\geq \int_t^\frac{t+1}{1}W(r)dr\approx (1-t)W(t).$$
Therefore, $W\in\R$ and
$$W(S_z)^\frac{p+q^\p}{pq^\p}\approx \frac{\om(S_z)\eta(S_z)^\frac{1}{p}}{\upsilon(S_z)^\frac{1}{q}}.
$$
So, for any $f\in A_\eta^p$ and $h\in A_{\sigma_{q,\upsilon}}^{q^\p}$,
by  (\ref{th1-r1}), (\ref{0410-1}), Theorem \ref{thA} and H\"older's inequality,
\begin{align*}
\int_{\D} |h(z)f(z)|d\mu(z)
&\lesssim M_0\left( \int_\D |f(z)h(z)|^\frac{pq^\p}{p+q^\p} W(z)dA(z) \right)^\frac{p+q^\p}{pq^\p}
\leq M_0\|f\|_{A_\eta^p}\|h\|_{A_{\sigma_{q,\upsilon}}^{q^\p}}.
\end{align*}
By Fubini's theorem and $(A_{\sigma_{q,\upsilon}}^{q^\p})^*\simeq A_\upsilon^{q}$ via $A_\om^2$-pairing, we have
\begin{align*}
\|\T_\mu^{\om} f\|_{A_\upsilon^q}
\approx\sup_{h\in A_{\sigma_{q,\upsilon}}^{q^\p}}\frac{\left|\langle h,\T_\mu^{\om} f\rangle_{A_{\om}^2} \right|}{\|h\|_{A_{\sigma_{q,\upsilon}}^{q^\p}}}
=\sup_{h\in A_{\sigma_{q,\upsilon}^{q^\p}}} \frac{\left|\int_{\D} h(z)\overline{f(z)}d\mu(z) \right|}{\|h\|_{A_{\sigma_{q,\upsilon}}^{q^\p}}}
\lesssim M_0\|f\|_{A_\eta^p},
\end{align*}
i.e.,
$\|\T_\mu\|\lesssim M_0.$

$(ii)$. Suppose  $1<p\leq q<\infty$  and $\T_\mu^\om:A_\eta^p\to A_\upsilon^q$ is compact. Let
$$b_z(w)=\frac{B_z^\om(w)}{\|B_z^\om\|_{A_\eta^p}}.$$
By (\ref{th1-2}) and $\sigma_{p,\eta}\in\R$, $\{b_z\}$ is bounded in $A_\eta^p$ and converges to 0 uniformly on compact subsets of $\D$ as $|z|\to 1$.
By Lemma \ref{lm-compact}, $\lim\limits_{|z|\to 1}\|\T_\mu^\om b_z\|_{A_\upsilon^q}=0$.
Then,  ({\ref{th1-r2}}) is obtained from (\ref{0410-1}), (\ref{th1-2}),
\begin{align*}
|\T_\mu^{\om} b_z(z)|
=\left|\langle \T_\mu^{\om} b_z, B_z^{\upsilon}   \rangle_{A_{\upsilon}^2}\right|
\leq \|\T_\mu^{\om}b_z\|_{A_\upsilon^q}   \|B_z^\upsilon\|_{A_\upsilon^{q^\p}}
\approx \|\T_\mu^{\om}b_z\|_{A_\upsilon^q}   \frac{1}{\upsilon(S_z)^\frac{1}{q}}
\end{align*}
and
\begin{align*}
| \T_\mu^{\om} b_z (z)|
&=\frac{1}{\|B_z\|_{A_\eta^p}}\int_\D |B_z^{\om}(\xi)|^2 d\mu(\xi)
\gtrsim  \frac{\om(D(z,r))}{\eta(D(z,r))^\frac{1}{p}}\frac{\mu(D(z,r))}{\om(D(z,r))^2}.
\end{align*}
Here, $r$ is that in (\ref{0315-2}).

Conversely, suppose (\ref{th1-r2}) holds.
For any $s\in (0,1)$, let $$d\mu_s=\chi_s(z)d\mu(z),~~~~~~~~~~~~ d\mu_{s,-}=d\mu-d\mu_s$$ and
$$M_s=\sup_{z\in\D}\frac{\mu_s(D(z,r))}{\om((D(z,r)))}   \frac{\upsilon(D(z,r))^\frac{1}{q}}{\eta(D(z,r))^\frac{1}{p}}.$$
Here, $\chi_s(z)$ is the  characteristic function of $\{z\in\D:|z|\geq s\}$.
Then, $\lim\limits_{s\to 1} M_s=0.$
Let $\{f_n\}$ be bounded in $A_\eta^p$ and converge to 0 uniformly on compact subsets of $\D$ as $n\rightarrow\infty$.
 Then, by statement $(i)$, we have
\begin{align*}
\lim_{n\to\infty}\|\T_\mu^\om f_n\|_{A_\upsilon^q}
&\leq \lim_{n\to\infty}\|\T_{\mu_{s,-}}^\om f_n\|_{A_\eta^q}+\lim_{n\to\infty}\|\T_{\mu_s}^\om f_n\|_{A_\eta^q}\\
&\lesssim C(s)\lim_{n\to\infty} \sup_{|z|<s}|f_n(z)|+M_s\|f_n\|_{A_\eta^p}
=M_s\|f_n\|_{A_\eta^p}.
\end{align*}
Letting $s\to 1$, $\lim\limits_{n\to\infty}\|\T_\mu^\om f_n\|_{A_\upsilon^q}=0$. By Lemma \ref{lm-compact}, $\T_\mu^\om:A_\eta^p\to A_\upsilon^q$ is compact.

$(iii)$.   $(iiia) \Rightarrow (iiib)$. It  is obvious.

$(iiib) $$\Rightarrow$ $(iiic)$.
Suppose $\T_\mu^{\om}:A_\eta^p\to A_\upsilon^q$ is bounded.
Let $\{z_j\}_{j=1}^\infty\subset\D$ be separated and $r$-covering.
Firstly, we choose $r\in(0,\infty)$ as that in (\ref{0315-2}).
For any $c=\{c_j\}_{j=1}^\infty\in l^p$ and $r_j(t)=$sign$(\sin (2^j\pi t))$, let
$$F_t(z)=\sum_{j=1}^\infty \frac{c_j r_j(t) B_{z_j}^{\om}}{\|B_{z_j}^{\om}\|_{A_\eta^p}}.$$
Then, Lemma \ref{lm2} implies
$$\|\T_\mu^{\om} F_t\|_{A_{\upsilon}^q}^q\lesssim \|\T_\mu^{\om}\|^q  \|c\|_{l^p}^q.$$
By Khinchin's inequality, we have
\begin{align*}
\int_0^1 \|\T_\mu^{\om} F_t\|_{A_{\upsilon}^q}^q dt
&=\int_\D \left(\int_0^1 \left| \sum_{j=1}^\infty \frac{c_j r_j(t) \T_\mu^{\om}B_{z_j}^{\om}(z) }{\|B_{z_j}^{\om}\|_{A_{\eta}^p}}\right|^q dt\right)
\upsilon(z)dA(z) \\
&\approx \int_\D \left(
\sum_{j=1}^\infty \frac{|c_j|^2|\T_\mu^{\om}B_{z_j}^{\om}(z) |^2}{\|B_{z_j}^{\om}\|_{A_{\eta}^p}^2}
\right)^\frac{q}{2} \upsilon(z)dA(z) \\
&\gtrsim \sum_{j=1}^\infty \frac{|c_j|^q }{\|B_{z_j}^{\om}\|_{A_{\eta}^p}^q}
\int_{D(z_j,r)}|\T_\mu^{\om}B_{z_j}^{\om}(z) |^q\upsilon(z)dA(z).
\end{align*}
Then, by (\ref{0410-1}),  (\ref{th1-2}), (\ref{0410-3}) and subharmonicity of $|\T_\mu^{\om}B_{z_j}^{\om}|^q$, we have
\begin{align*}
\int_0^1 \|\T_\mu^{\om} F_t\|_{A_{\upsilon}^q}^q dt
&\gtrsim \sum_{j=1}^\infty |c_j|^q
\left( \frac{\om(D(z_j,r))\upsilon(D(z_j,r))^\frac{1}{q}}{\eta(D(z_j,r))^\frac{1}{p} }
|\T_\mu^{\om}B_{z_j}^{\om}(z_j)| \right)^q\\
&\gtrsim \sum_{j=1}^\infty |c_j|^q \left(
\frac{\mu(D(z_j,r))}{\om(D(z_j,r))}
\frac{\upsilon(D(z_j,r))^\frac{1}{q}}{\eta(D(z_j,r))^\frac{1}{p} }
\right)^q.
\end{align*}
So, for all $c=\{c_j\}_{j=1}^\infty\in l^p$, we have
$$\sum_{j=1}^\infty |c_j \lambda_j|^q\lesssim \|\T_\mu^\om\|^q \, \|c\|_{l^p}^q.$$
The classical duality relation $(l^\frac{p}{q})^*\simeq l^\frac{p}{p-q}$ implies that $\{\lambda_j\}\in l^\frac{pq}{p-q}$
and $$\|\{\lambda_j\}\|_{l^\frac{pq}{p-q}}\lesssim\|\T_\mu\|.$$

Suppose $r_0>r$ and   $\{z_j^\p\}_{j=1}^\infty$ is separated and $r_0$-covering.
Let $\chi(j,i)=1$ when $D(z_j^\p,r_0)\cap D(z_i,r)\neq \emptyset$ and $\chi(i,j)=0$ otherwise.
It is easy to check that there exists a natural number $K$ such that
$$1\leq \sum_{j=1}^\infty \chi(j,i)\leq K,\,\,\,\,1\leq \sum_{i=1}^\infty\chi(j,i)\leq K.$$
Then
\begin{align*}
\sum_{j=1}^\infty \left|\frac{\mu(D(z_j^\p,r_0))}{\om(D(z_j^\p,r_0))}
         \frac{\upsilon(D(z_j^\p,r_0))^\frac{1}{q}}{\eta(D(z_j^\p,r_0))^\frac{1}{p} }\right|^\frac{pq}{p-q}
&\lesssim  \sum_{j=1}^\infty \left|  \sum_{i=1}^\infty \frac{\mu(D(z_i,r))}{\om(D(z_i,r))}
         \frac{\upsilon(D(z_i,r))^\frac{1}{q}}{\eta(D(z_,r))^\frac{1}{p} } \chi(j,i) \right|^\frac{pq}{p-q} \\
&\approx \sum_{i=1}^\infty \sum_{j=1}^\infty \left|  \frac{\mu(D(z_i,r))}{\om(D(z_i,r))}
         \frac{\upsilon(D(z_i,r))^\frac{1}{q}}{\eta(D(z_,r))^\frac{1}{p} } \chi(j,i) \right|^\frac{pq}{p-q} \\
         &\approx \sum_{i=1}^\infty  \left|  \frac{\mu(D(z_i,r))}{\om(D(z_i,r))}
         \frac{\upsilon(D(z_i,r))^\frac{1}{q}}{\eta(D(z_,r))^\frac{1}{p} }  \right|^\frac{pq}{p-q}. \\
\end{align*}
Therefore, $(iiic)$ holds and $\|\{\lambda_j\}\|_{l^\frac{pq}{p-q}}\lesssim\|\T_\mu^\om\|$ for all $0<r<\infty$.

$(iiic) \Rightarrow (iiid)$. Suppose $(iiic)$ holds.
Let $\{z_j\}_{j=1}^\infty$ be separated and $s$-covering.
For convenience, let
$$W_1(z)=\frac{\om(D(z,r))^{1+\frac{1}{q}-\frac{1}{p}}\eta(D(z,r))^\frac{1}{p}}{\upsilon(D(z,r))^\frac{1}{q}}.$$
Then
\begin{align*}
\|\widehat{\mu}_r\|_{L_\om^\frac{pq}{p-q}}^\frac{pq}{p-q}
&\lesssim \sum_{j=1}^\infty \int_{D(z_j,s)} \left(\frac{\mu(D(z,r))}{W_1(z)}\right)^\frac{pq}{p-q}\om(z)dA(z)   \\
&\lesssim \sum_{j=1}^\infty \left(\frac{\mu(D(z_j,r+s))}{W_1(z_j)}\right)^\frac{pq}{p-q}\om(D(z_j,r+s))\\
&\approx \sum_{j=1}^\infty \left( \frac{\mu(D(z_j,r+s))}{\om(D(z_j,r+s))}
\frac{\upsilon(D(z_j,r+s))^\frac{1}{q}}{\eta(D(z_j,r+s))^\frac{1}{p} } \right)^\frac{pq}{p-q} .
\end{align*}
Therefore, $\widehat{\mu}_r\in L_\om^\frac{pq}{p-q}$
and $$\|\widehat{\mu}_r\|_{L_\om^\frac{pq}{p-q}}\lesssim \|\{\lambda_j\}\|_{l^\frac{pq}{p-q}}.$$

$(iiid)\Rightarrow(iiib)$. Suppose $(iiic)$ holds.
Let $W$ be defined as that in (\ref{0317-1}).
Then we have
$$\int_\D\left(\frac{\mu(D(z,r))}{W(D(z,r))}\right)^\frac{pq}{p-q}W(z)dA(z)
\approx\|\widehat{\mu}_r\|_{L_\om^\frac{pq}{p-q}}^\frac{pq}{p-q}<\infty. $$
So, by Theorem \ref{thA}, $Id:A_{W}^{\frac{pq}{pq+q-p}}\to L_\mu^1$ is bounded and
$$\|I_d\|_{A_{W}^\frac{pq}{pq+q-p}\to L_\mu^1}\approx \|\widehat{\mu}_r\|_{L_\om^\frac{pq}{p-q}}.$$
Therefore, for any $f\in{A_{\eta}^p}$ and $h\in {A_{\sigma_{q,\upsilon}}^{q^\p}}$,
 by Fubini's theorem and H\"older's inequality,  we have
\begin{align*}
|\langle h,\T_\mu^\om f\rangle_{A_\om^2}|
&\leq \int_\D |h(z)f(z)|d\mu(z)\\
&\leq \|\widehat{\mu}_r\|_{L_\om^\frac{pq}{p-q}}\left(\int_\D |f(z)h(z)|^\frac{pq}{pq+q-p}W(z)dA(z)\right)^\frac{pq+q-p}{pq}\\
&\lesssim \|\widehat{\mu}_r\|_{L_\om^\frac{pq}{p-q}}\|f\|_{A_{\eta}^p} \|h\|_{A_{\sigma_{q,\upsilon}}^{q^\p}}.
\end{align*}
Then, (\ref{th1-1}) implies  $\|\T_\mu^\om\|\lesssim \|\widehat{\mu}_r\|_{L_\om^\frac{pq}{p-q}}.$

$(iiib)\Rightarrow(iiia)$. By \cite[Theorem 3.2]{ZxXlFhLj2014amsc}, $A_\eta^p$ and $A_\upsilon^q$ are isomorphic to $l^p$ and $l^q$, respectively.  Theorem I.2.7 in \cite{LjTl} shows that every bounded operator from $l^p$ to $l^q$ is compact when $0<q<p<\infty$.
Therefore, $(iiib)\Rightarrow(iiia)$.
The proof is complete.
\end{proof}

\begin{proof}[Proof of Theorem 2]
Suppose $\mu$ is a $\lambda$-Carleson measure for $A_\om^1$  and $n\geq 2$. 
Let $h_i\in A_{\om_i}^{p_i/{q_i}}(i=1,2,\cdots,n)$. By H\"older's inequality,
\begin{align*}
\left\|\prod_{i=1}^n h_i\right\|_{A_\om^{1/\lambda}}
=\left(\int_\D \prod_{i=1}^n \left(|h_i(z)|^{\frac{1}{\lambda}}\om_i(z)^\frac{q_i}{\lambda p_i}\right) dA(z)\right)^\lambda
\leq \prod_{i=1}^n\|h_i\|_{A_{\om_i}^{p_i/q_i}}.
\end{align*}
Then, letting $C_{\mu,\om,\lambda}=\|I_d\|_{A_\om^{1/\lambda}\to L_\mu^1}$, we have
\begin{align}\label{th2-1}
\left\|\prod_{i=1}^n h_i\right\|_{L_\mu^1}
\leq  C_{\mu,\om,\lambda} \left\|\prod_{i=1}^n h_i\right\|_{A_\om^{1/\lambda}}
\leq C_{\mu,\om,\lambda}\prod_{i=1}^n\|h_i\|_{A_{\om_i}^{p_i/q_i}}.
\end{align}
Let
$$d\mu_1(z)=\prod_{i=2}^n|h_i(z)|d\mu(z).$$
Then,  Theorem \ref{thA} and (\ref{th2-1}) imply
$$\|I_d\|_{A_{\om_1}^{p_1}\to L_{\mu_1}^{q_1}}^{q_1}
\approx \|I_d\|_{A_{\om_1}^{p_1/q_1}\to L_{\mu_1}^1}
\lesssim  C_{\mu,\om,\lambda}\prod_{i=2}^n\|h_i\|_{A_{\om_i}^{p_i/q_i}}.$$
Therefore, for all $f_1\in A_\om^{p_1}$,
\begin{align}\label{th2-2}
\int_\D |f_1(z)|^{q_1}d\mu_1(z)
\leq \|I_d\|_{A_{\om_1}^{p_1}\to L_{\mu_1}^{q_1}}^{q_1}\|f_1\|_{A_{\om_1}^{p_1}}^{q_1}
\lesssim  C_{\mu,\om,\lambda}\prod_{i=2}^n\|h_i\|_{A_{\om_i}^{p_i/q_i}}  \|f_1\|_{A_{\om_1}^{p_1}}^{q_1}.
\end{align}
Similarly, letting
$$d\mu_2(z)=|f_1(z)|^{q_1}|h_3(z)h_4(z)\cdots h_n(z)|d\mu(z),$$
by Theorem \ref{thA} and (\ref{th2-2}), we have
$$\|I_d\|_{A_{\om_2}^{p_2}\to L_{\mu_2}^{q_2}}^{q_2}
\approx \|I_d\|_{A_{\om_2}^{p_2/q_2}\to L_{\mu_2}^1}
\lesssim C_{\mu,\om,\lambda}\|f_1\|_{A_{\om_1}^{p_1}}^{q_1}\prod_{i=3}^n\|h_i\|_{A_{\om_i}^{p_i/q_i}},$$
and, for all $f_2\in A_{\om_2}^{p_2}$,
\begin{align*}
\int_\D |f_2(z)|^{q_2}d\mu_2(z)
\leq& \|I_d\|_{A_{\om_2}^{p_2}\to L_{\mu_2}^{q_2}}^{q_2}\|f_2\|_{A_{\om_2}^{p_2}}^{q_2}\\
\lesssim & C_{\mu,\om,\lambda}\|f_1\|_{A_{\om_1}^{p_1}}^{q_1}  \|f_2\|_{A_{\om_2}^{p_2}}^{q_2} \prod_{i=3}^n\|h_i\|_{A_{\om_i}^{p_i/q_i}}.
\end{align*}
Continuing this process, we   get (\ref{th2-r})
and $M_n\lesssim \|I_d\|_{A_\om^{1/\lambda}\to L_\mu^1}.$

Conversely, suppose (\ref{th2-r}) holds. When $\lambda\geq 1$, by Lemma 2.4 in \cite{PjaRj2014book}, we can choose $\gamma$ large enough such that
$$\|F_a\|_{A_{\om_i}^{p_i}}^{p_i}\approx \om_i(S_a),\,\,\,\mbox{ for all }\,\,\,i=1,2,\cdots,n \,\,\,\mbox{ and }\,\,\,a\in\D,$$
where $F_a(z)=\left(\frac{1-|a|^2}{1-\overline{a}z}\right)^\gamma$.
For any $z\in S_a$ and $a\in\D$, we have $|1-\overline{a}z|\approx 1-|a|$.
So, (\ref{th2-r}) implies
\begin{align*}
\mu(S_a)\lesssim
\int_\D \left|\frac{1-|a|^2}{1-\overline{a}z}\right|^{\gamma(q_1+q_2+\cdots+q_n)}d\mu(z)
\lesssim  M_n \om(S_a)^\lambda.
\end{align*}
By Theorem \ref{thA}, $\mu$ is a $\lambda$-Carelson measure for $A_\om^1$ and
$M_n \gtrsim \|I_d\|_{A_\om^{1/\lambda}\to L_\mu^1}.$

So, we only need to prove the case of $0<\lambda<1$. We use induction on $n$. When $n=1$, it is just the definition of Carleson measure for Bergman spaces.   Now, let $n\geq 2$ and the result holds for $n-1$ functions. Set
$$\lambda_{n-1}=\sum_{i=1}^{n-1}\frac{q_i}{p_i},
\,\,\,\,\,\eta(z)=\prod\limits_{i=1}^{n-1} \om_i(z)^\frac{q_i}{\lambda_{n-1} p_i},
\,\,\,\,\,d\mu_{f_n,q_n}(z)=|f_n(z)|^{q_n}d\mu(z),$$
and
$$
M_{n-1,f_n,q_n}=\sup_{f_i\in A_{\om_i}^{p_i},i=1,2,\cdots,n-1}
\frac {\int_\D \prod_{i=1}^{n-1}|f_i(z)|^{q_i}d\mu_{f_n,q_n}(z)} {\prod_{i=1}^{n-1}\|f_i\|_{A_{\om_i}^{p_i}}^{q_i}}.
$$
Then, (\ref{th2-r}) and the assumption imply
$$ \|I_d\|_{A_\eta^{{1/{\lambda_{n-1}}}}\to L_{\mu_{f_n,q_n}}^1} \approx M_{n-1,f_n,q_n} \lesssim M_n \|f_n\|_{A_{\om_n}^{p_n}}^{q_n}.$$
Since $\lambda_{n-1}<1$, by Theorem \ref{thA}, if $\gamma$ is large enough and fixed, we have
{\small
\begin{align}\label{th2-3}
S(f_n):=&\int_\D\left(\int_\D \left(\frac{1-|\xi|}{|1-\overline{z}\xi|}\right)^{\gamma}\frac{|f_n(\xi)|^{q_n}d\mu(\xi)}{(1-|\xi|)^2\eta(\xi)} \right)^\frac{1}{1-\lambda_{n-1}}\eta(z)dA(z) \nonumber\\
 \lesssim& M_n^\frac{1}{1-\lambda_{n-1}} \|f_n\|_{A_{\om_n}^{p_n}}^{\frac{q_n}{1-\lambda_{n-1}}}.
\end{align}
}
Let $\delta$ be large enough, $\{a_k\}_{k=1}^\infty$ be separated and $r$-covering,
$$dV_z(\xi)=\left(\frac{1-|\xi|}{|1-\overline{z}\xi|}\right)^{\gamma}\frac{d\mu(\xi)}{(1-|\xi|)^2\eta(\xi)},
\,\,\,\,b_{a_k}=\frac{1}{\om_n(S_{a_k})^\frac{1}{p_n}}\left(\frac{1-|a_k|}{1-\overline{a_k}z}\right)^\delta.$$
Then, for any $c=\{c_k\}\in l^{p_n}$, from Theorem 3.2 in \cite{ZxXlFhLj2014amsc}, we have
 $$\|f_t\|_{A_{\om_n}^{p_n}}\lesssim \|\{c_k\}\|_{l^{p_n}},$$
 where
$f_t(\xi)=\sum_{k=1}^\infty c_k r_k(t)b_{a_k}(\xi).$
Then, by Fubini's theorem, Kahane's inequality and Khinchin's inequality, we have
\begin{align*}
\int_0^1 S(f_t) dt=&\int_\D \int_0^1  \left\|\sum_{k=1}^\infty c_k r_k(t)b_{a_k}(\xi)\right\|_{L_{V_z}^{q_n}}^\frac{q_n}{1-\lambda_{n-1}}dt \eta(z)dA(z) \\
\approx &\int_\D \left(\int_0^1  \left\|\sum_{k=1}^\infty c_k r_k(t)b_{a_k}(\xi)\right\|_{L_{V_z}^{q_n}}^{q_n}dt \right)^\frac{1}{1-\lambda_{n-1}}\eta(z)dA(z) \\
\approx& \int_\D \left(\int_\D \left(\sum_{k=1}^\infty |c_k|^2 |b_{a_k}(\xi)|^2 \right)^\frac{q_n}{2}dV_z(\xi)
\right)^\frac{1}{1-\lambda_{n-1}}\eta(z)dA(z)\\
\gtrsim& \sum_{j=1}^\infty\int_{D(a_j,r)} \left(\int_{D(a_j,r)} \left( |c_j|^2 |b_{a_j}(\xi)|^2 \right)^\frac{q_n}{2}dV_z(\xi)
\right)^\frac{1}{1-\lambda_{n-1}}\eta(z)dA(z)\end{align*}
 \begin{align*}
\approx&\sum_{j=1}^\infty |c_j|^\frac{q_n}{1-\lambda_{n-1}} \frac{\mu(D(a_j,r))^\frac{1}{1-\lambda_{n-1}}}{\om_n(S_{a_j})^\frac{q_n}{p_n(1-\lambda_{n-1})}\eta(D(a_j,r))^\frac{\lambda_{n-1}}{1-\lambda_{n-1}}}
\\
=&\sum_{j=1}^\infty |c_j|^\frac{q_n}{1-\lambda_{n-1}} \left(\frac{\mu(D(a_j,r))}{\om(D(a_j,r))^\lambda}\right)^\frac{1}{1-\lambda_{n-1}}.
\end{align*}
From (\ref{th2-3}) and the classical duality relation
$$\left(l^\frac{p_n(1-\lambda_{n-1})}{q_n}\right)^*\simeq l^\frac{1-\lambda_{n-1}}{1-\lambda},$$ $\left\{\left(\frac{\mu(D(a_k,r))}{\om(D(a_k,r))^\lambda}\right)^\frac{1}{1-\lambda_{n-1}}\right\}$ defines a bounded functional on $l^\frac{p_n (1-\lambda_{n-1})}{q_n}$ and
$$
 \left\| \left\{\frac{\mu(D(a_k,r))}{\om(D(a_k,r))^\lambda}\right\}\right\|_{l^\frac{1}{1-\lambda}}^\frac{1}{1-\lambda_{n-1}}
= \left\| \left\{\left(\frac{\mu(D(a_k,r))}{\om(D(a_k,r))^\lambda}\right)^\frac{1}{1-\lambda_{n-1}}\right\}\right\|_{l^\frac{1-\lambda_{n-1}}{1-\lambda}}
\lesssim M_n^\frac{1}{1-\lambda_{n-1}}.
$$
Then, by Theorem \ref{thA},  we have
\begin{align*}
\|I_d\|_{A_\om^{1/\lambda}\to L_\mu^1}
&\approx \left(\int_\D \left( \frac{\mu(D(z,r))}{\om(D(z,r))}\right)^\frac{1}{1-\lambda}\om(z)dA(z)\right)^{1-\lambda}\\
&\approx \left(\sum_{k=1}^\infty \frac{\mu(D(a_k,r))^\frac{1}{1-\lambda}}{\om(D(a_k,r))^\frac{\lambda}{1-\lambda}}\right)^{1-\lambda}
\lesssim M_n.
\end{align*}
The proof is complete.
\end{proof}

The proof of Theorem 3 can be deduced in a standard way, see the proof of Theorem 4.1 in \cite{PjZr2015mmj} for example. For the benefits of readers, we prove it here.

\begin{proof}[Proof of Theorem 3]
$(i)\Rightarrow(ii)$. Suppose $(i)$ holds, i.e., $\mu$ is a vanishing $\lambda$-Carleson measure for $A_\om^1$.
If $s\in[0,1)$,  let $\chi_s(z)$ be the  characteristic function of $\{z\in\D:|z|\geq s\}$ and $d\mu_s(z)=\chi_s(z)d\mu(z)$.
By Theorem A, we have
$$\lim_{s\to 1} \|I_d\|_{A_\om^{1/\lambda}\to L_{\mu_s}^1}=\lim\limits_{s\to 1}\|I_d\|_{A_\om^1\to L_{\mu_s}^\lambda}^\lambda=0.$$
By Theorem \ref{th2}, we have
$$
\lim_{s\to 1}\sup_{f_i\in A_{\om_i}^{p_i},i=1,2,\cdots,n}\frac{\int_\D \prod_{i=1}^n |f_i(z)|^{q_i} d\mu_s(z)}{\prod_{i=1}^n \|f_i\|_{A_{\om_i}^{p_i}}^{q_i}}=0.
$$
So, for any given $\varepsilon>0$, there exists a real number $s_0\in(0,1)$ such that
\begin{align}\label{0424-1}
\sup_{f_i\in A_{\om_i}^{p_i},i=1,2,\cdots,n}\frac{\int_\D \prod_{i=1}^n |f_i(z)|^{q_i} d\mu_{s_0}(z)}{\prod_{i=1}^n \|f_i\|_{A_{\om_i}^{p_i}}^{q_i}}<\varepsilon.
\end{align}
Meanwhile, when $i=2,\cdots,n$,  for any $\|f_i\|_{A_{\om_i}^{p_i}}\leq 1$ and $|z|\leq s_0$, it is easy to check that
$$|f_i(z)|\lesssim \frac{\|f_i\|_{A_{\om_i}^{p_i}}}{(1-|z|)^\frac{1}{p_i}\widehat{\om_i}(z)^\frac{1}{p_i}}
\leq \frac{1}{(1-s_0)^\frac{1}{p_i}\widehat{\om_i}(s_0)^\frac{1}{p_i}}.$$
Since $\{f_{1,k}\}_{k=1}^\infty$ converges to 0 uniformly on $s_0\D$, there exists a natural number $K$ such that, for all $k>K$,
\begin{align}\label{0424-2}
\int_{s_0\D} |f_{1,k}(z)|^{q_1}\prod_{i=2}^n |f_i(z)|^{q_i}d\mu(z)<\varepsilon.
\end{align}
 Then, (\ref{0424-1}) and (\ref{0424-2}) imply  $\lim\limits_{k\to\infty} F(k)=0$, i.e.,  $(ii)$ holds.

$(ii)\Rightarrow(iii)$. It is obvious.

$(iii)\Rightarrow(i)$. Suppose $(iii)$ holds.  Then, $\mu$ is a $\lambda$-Carleson measure for $A_\om^1$.
 Otherwise, by Theorem \ref{th2}, there exist  $\{f_{1,k}\}_{k=1}^\infty$, $\cdots$, $\{f_{n,k}\}_{k=1}^\infty$ in the unit ball of $A_{\om_1}^{p_1}$, $\cdots$, $A_{\om_2}^{p_n}$, respectively, such that
 $$\int_{\D} \prod_{i=1}^n \left|\frac{f_{i,k}(z)}{k}\right|^{q_i} d\mu(z)>1.$$
Meanwhile, since  $\lim\limits_{k\to \infty}\|\frac{f_{i,k}}{k}\|_{A_{\om_i}^{p_i}}=0$, by statement $(iii)$, we have
$$\int_{\D} \prod_{i=1}^n \left|\frac{f_{i,k}(z)}{k}\right|^{q_i} d\mu(z)=0.$$
This is a contradiction. So, $\mu$ is a $\lambda$-Carleson measure for $A_\om^1$.
Moreover, when $0<\lambda<1$, by Theorem \ref{thA}, $I_d:A_\om^1\to L_\mu^\lambda$ is   compact.

Suppose $\lambda>1$.  By Lemma 2.4 in \cite{PjaRj2014book}, we can choose $\gamma$ large enough such that
$$\|F_a\|_{A_{\om_i}^{p_i}}^{p_i}\approx \om_i(S_a),\,\,\,\mbox{ for all }\,\,\,i=1,2,\cdots,n \,\,\,\mbox{ and }\,\,\,a\in\D,$$
where $F_a(z)=\left(\frac{1-|a|^2}{1-\overline{a}z}\right)^\gamma$.
Therefore, for $i=1,2,\cdots,n$, $f_{i,a}(z)=\frac{F_a(z)}{\|F_a\|_{A_{\om_i}^{p_i}}}$  is bounded in $A_{\om_i}^{p_i}$ and converges to 0 uniformly on compact subsets of $\D$ as $|a|\to 1$. From $(iii)$, we have
\begin{align}\label{0424-3}
\lim_{|a|\to 1} \int_\D \left(\frac{1-|a|}{|1-\overline{a}z|}\right)^{(q_1+q_2+\cdots+q_n)\gamma}\frac{d\mu(z)}{\om(S_a)^\lambda}=0.
\end{align}
For brief, let $\gamma_n=(q_1+q_2+\cdots+q_n)\gamma$. Then, for any fixed $r>0$, by (\ref{0410-1})  we have
\begin{align} \frac{\mu(D(a,r))}{\om(D(a,r))^\lambda}
\approx &\int_{D(a,r)}\left(\frac{1-|a|}{|1-\overline{a}z|}\right)^{\gamma_n}\frac{d\mu(z)}{\om(S_a)^\lambda} \nonumber \\
\leq &\int_\D\left(\frac{1-|a|}{|1-\overline{a}z|}\right)^{\gamma_n}\frac{d\mu(z)}{\om(S_a)^\lambda}. \nonumber \end{align}
Therefore, (\ref{0424-3}) and Theorem \ref{thA} imply that $\mu$ is a vanishing $\lambda$-Carleson measure for $A_\om^1$. The proof is complete.

\end{proof}

\end{document}